\documentclass[12pt,reqno]{elsarticle}

\usepackage{amsfonts,color,amsmath,amssymb,fancyhdr}
\usepackage{amsthm}
\usepackage{tikz}
\usepackage{graphicx}
\usepackage{subfigure}
\def\Box{\vcenter{\vbox{\hrule\hbox{\vrule
     \vbox to 8.8pt{\hbox to 10pt{}\vfill}\vrule}\hrule}}}

\newcommand{\tr}{\textup{Tr}}

\newcommand{\PGaL}{\textup{P}\Gamma \textup{L}}

\newcommand{\F}{{\mathbb F}}

\newcommand{\cN}{{\mathcal N}}

\newcommand{\cU}{{\mathcal U}}

\newcommand{\cS}{{\mathcal S}}

\newcommand{\PG}{\textup{PG}}

\newcommand{\Aut}{\textup{Aut}}

\newtheorem{thm}{Theorem}

\newtheorem{lemma}[thm]{Lemma}

\numberwithin{equation}{section}
\numberwithin{thm}{section}

\theoremstyle{definition}

\begin{document}
\newcommand{\stopthm}{\begin{flushright}
\(\box \;\;\;\;\;\;\;\;\;\; \)
\end{flushright}}

\newcommand{\symfont}{\fam \mathfam}

\title{On the existence of O'Nan configurations in ovoidal Buekenhout-Metz unitals in $\PG(2,q^2)$}
\author[add1]{Tao Feng} \ead{tfeng@zju.edu.cn}
\author[add1]{Weicong Li\corref{cor1}}\ead{conglw@zju.edu.cn}
\cortext[cor1]{Corresponding author}
\address[add1]{School of Mathematical Sciences, Zhejiang University, 38 Zheda Road, Hangzhou 310027, Zhejiang P.R China}

\begin{abstract}

In this paper, we establish the existence of O'Nan configurations in all nonclassical ovoidal Buekenhout-Metz unitals in  $\PG(2,q^2)$.

 \vspace*{3mm}

\end{abstract}

\maketitle

\section{Introduction}
A unital of order $n$ is a design with parameters $2-(n^3+1,n+1,1)$. All the known unitals have order a prime power except one of order $6$ constructed in \cite{BB} and \cite{Mathon6}. The classical unital of order $q$, $q$ a prime power, consists of the absolute points and non-absolute lines of a unitary polarity in the Desarguesian plane $\PG(2,q^2)$.  A unital $\cU$ embedded in a projective plane $\Pi$ of order $q^2$ is a set  of $q^3+1$ points of $\Pi$ such that each line meets $\cU$ in exactly  $1$ or $q+1$ points. A line that intersects the unital $\cU$ in  $1$ or $q+1$ points is called a tangent or secant line respectively. The blocks of $\cU$ consist of the intersections of $\cU$ with the secant lines.

In 1976, Buekenhout \cite{Buekenhout1976Existence} used the Bruck-Bose model to show that each two-dimensional translation plane contains a unital. The unitals arising from this construction are called Buekenhout unitals. Metz \cite{Metz1979On} showed that there are  nonclassical Buekenhout unitals in $\PG(2,q^2)$ for any prime power $q>2$. All the known unitals in finite Desraguesian planes are Buekenhout unitals. Please refer to the monograph \cite{Ebert2008Unitals} for more information.

In 1972, O'Nan \cite{O1972Automorphisms} observed that the classical unitals contain no {\bf O'Nan configuration}, which is a configuration consisting of  four distinct lines intersecting in six distinct points of $\cU$. In \cite{PiperUnitary} Piper  conjectured that the absence of O'Nan configurations characterizes the classical unitals. In \cite{MR690826} Wilbrink gave an intrinsic characterization of the classical unitals assuming the absence of such a configuration and two further conditions. In \cite{Hui2014On}, the authors gave a necessary and sufficient condition for a unital to be embedded in a projective plane and strengthened Wilbrink's results based on the characterization results in \cite{MR690826} and \cite{GSV}. On the other hand, the existence of O'Nan configurations in certain unitals embedded  in non-Desarguesian projective planes has been established in \cite{Hui2013on,Tai2014on}.

In this paper, we consider the existence of O'Nan configurations in Buekenhout unitals in $\PG(2,q^2)$. In Section 2, we introduce some backgrounds and preliminary results. In Section 3, we establish the existence of an O'Nan configuration in each nonclassical orthogonal Bukenhout-Metz unital $\cU$, where the configuration is fixed by an involution in the stabilizer of $\cU$ in $\textup{P}\Gamma\textup{L}(3,q^2)$. In section 4, we establish the existence of O'Nan configurations of a particular form that contains a fixed Baer subline in Buekenhout-Tits unitals.

\section{Preliminaries}
Throughout this paper, we fix the following notation. Let $p$ be a prime and $m$ be a positive integer such that $q:=p^m$ is larger than $2$. For a divisor $d$ of $m$, define the trace function as
\[
\tr_{q/p^d}(x)=x+x^{p^d}+\cdots+x^{p^{m-d}},\quad\textup{ for } x\in\F_{q}.
\]
For a nonzero vector $u\in\F_q^3$, we define
\[
[u]:=\{x\in\PG(2,q):\,x\cdot u=0\},
\]
where $\cdot$ is the usual Euclidean inner product. It is clear that $[u]$ is a line of $\PG(2,q)$.

Let us briefly recall Buekenhout's construction of unitals from ovoidal cones.
Let $\Sigma_{\infty}$ be a hyperplane of the projective $4$-space $\PG(4,q)$, and suppose that $\Sigma_\infty$ contains a spread $\cS$. We can define an incidence structure $\Pi$ as follows: the points of $\Pi$ are the points in $\PG(4,q)\setminus\Sigma_{\infty}$  and the spread lines of $\mathcal{S}$, the lines of $\Pi$ are the planes in $\PG(4,q)$ intersecting $\Sigma_{\infty}$ in a line of $\cS$ together with $\Sigma_{\infty}$, and  incidence is by inclusion.  Then $\Pi$ is a translation plane, and this is known as the Bruck-Bose model \cite{Bruck1963The,R1966Linear}. Take an ovoidal cone $U$ of $\PG(4,q)$ that meets $\Sigma_\infty$ in a line $\ell$ of $\cS$. Then the set $\cU$ of  points of $\Pi$ contained in $U$ forms a unital in $\Pi$, called an \textit{ovoidal Buekenhout-Metz unital}. We call $\ell$ the \textit{special point} of $\cU$.  In the case $U$ is an elliptic cone, the corresponding unital $\cU$ is an orthogonal Buekenhout-Metz unital; in the case $U$ is a cone over a Tits ovoid, $\cU$ is a Buekenhout-Tits unital. In the sequel, we will only consider the case $\Pi$ is a Desarguesian plane $\PG(2,q^2)$, i.e., $S$ is a regular spread.  There is a similar construction with $U$ replaced by a nonsingular quadric in $\PG(4,q)$. Barwick \cite{MR1298769} used a counting argument to show that the construction using nonsingular quadrics in $\PG(4,q)$ does not give new examples of unitals in $\PG(2,q^2)$.

In \cite{Baker1992On,Ebert1992On} Baker and Ebert derived the  expression for ovoidal Buekenhout-Metz unitals in $\PG(2,q^2)$, $q>2$. Each orthogonal Buekenhout-Metz unital in $\PG(2,q^2)$, $q>2$, is projectively equivalent to one of the following form:
\begin{equation}\label{def_BM}
\cU_{\alpha,\beta}=\{(x,\alpha\,x^2+\beta x^{q+1}+r,1): r\in \F_q\, , x\in \F_{q^2}\}\cup\{(0,1,0)\}
\end{equation}
where $\alpha,\beta$ are elements in $\F_{q^2}$ with the following properties:  $d= (\beta-\beta^q)^2+4\alpha^{q+1}$ is a nonsquare in $\F_q$ in the case $q$ is odd;  $\beta\not\in\F_q$, and $d=\frac{\alpha^{q+1}}{(\beta+\beta^q)^2}$ has absolute trace $0$ in the case $q$ is even and $q>2$. The quantity $d$ is called the \textit{discriminant} of the unital $\cU_{\alpha,\beta}$. The unital $\cU_{\alpha,\beta}$ is classical if and only if $\alpha=0$. The equivalence amongst such unitals is determined as follows.
\begin{thm}\label{BM_pj}\cite{Baker1992On, Ebert1992On}
In $\PG(2,q^2)$, $q>2$, two   unials  $\cU_{\alpha,\beta}$ and $\cU_{\alpha',\beta'}$ are projectively equivalent if and only if there exists $f\in \F_{q}^*$, $s\in \F_{q^2}^*$, $u\in \F_q$ and $\sigma \in \Aut(\F_{q^2})$ such that
\[
 \alpha'=\alpha^{\sigma}s^2f,\quad \beta'=\beta^{\sigma}s^{q+1}f+u.
\]
\end{thm}

We next describe the expression for the Buekenhout-Tits unitals. Let $q=2^m$ with $m$ an odd integer greater than $1$. Set 
\[
f(x_0,x_1)=x_0^{\tau+2}+x_0x_1+x_1^{\tau},\quad \tau:=2^{(m+1)/2}.
\]
By a similar procedure to that in \cite{Ebert1997Buekenhout} and some detailed analysis, we can show that a Buekenhout-Tits unital in $\PG(2,q^2)$ must be projectively equivalent to one of  the following form
\begin{equation} \label{def_BT}
   \cU_{T}=\{(0,1,0)\}\cup\{(x_0+x_1\delta,r+f(x_0,x_1)\delta,1):x_0,x_1,r \in \F_q \},
\end{equation}
where $\delta$ is an element of $\F_{q^2}\setminus\F_q$, and different choices of $\delta$ yields projectively equivalent unitals. In particular, the Buekenhout-Tits unital in $\PG(2,q^2)$ is unique up to projective equivalence, justifying the notation $\cU_T$. We do not find a reference for this uniqueness result in the literature, but we do not include a proof here due to its considerable similarity with \cite{Ebert1997Buekenhout}.\\

The following result is Theorem 12.8.7 in  \cite{PayneTopics}.
\begin{lemma}\label{fxy}
If $f(x,y)=x^{\tau+2}+xy+y^{\tau}\ne 0$ for $x,y\in\F_q$, then
$$\frac{1}{f(x,y)}=f\left(\frac{y}{f(x,y)},\frac{x}{f(x,y)}\right).$$
\end{lemma}
\begin{proof}
The claim follows by directly checking that
\begin{equation*} 
\begin{array}{rl}
f(x,y)^{\tau+1}=y^{\tau+2}+xy(f(x,y))^\tau+x^\tau(f(x,y))^2
\end{array}
\end{equation*}
and dividing both sides by $f(x,y)^{\tau+2}$.
\end{proof}
We shall also need the following result on O'Nan configurations.
\begin{lemma}\cite[Lemma 7.42]{Ebert2008Unitals} \label{Spoint}
Let $\cU$ be a ovoidal Buekenhout-Metz unital in $\PG(2,q^2)$, $q>2$. Then there is no O'Nan configuration that contains the special point of $\cU$.
\end{lemma}

\section{O'Nan configurations in orthogonal Buekenhout-Metz unitals}

In this section, we establish the following result.
\begin{thm}\label{BM_Onan}
Each nonclasical orthogonal Buekenhout-Metz unital in $\PG(2,q^2)$, $q>2$, contains an O'Nan configuration.
\end{thm}
Let $\cU_{\alpha,\beta}$ be a nonclassical orthogonal Buekenhout-Metz unital in $\PG(2,q^2)$, $q>2$, as defined in Eqn. \eqref{def_BM}. Here, we have $\alpha\ne 0$, and write $d$ as its determinant. Let $\cN$ be a putative O'Nan configuration in $\cU_{\alpha,\beta}$.  The subgroup of $\PGaL(3,q^2)$ that leaves $\cU_{\alpha,\beta}$ invariant and fixes the special points $P_{\infty}=(0,1,0)$ acts transitively on the points of $\cU_{\alpha,\beta}\setminus\{P_{\infty}\}$ by \cite[Theorem 4.12, 4.23]{Ebert2008Unitals}. Therefore we assume without of loss of generality that  $\cN$ contains the point  $P=(0,0,1)$.  Our strategy is to assume that  $\cN$ is stabilized by a properly chosen involution in the stabilizer of $\cU_{\alpha,\beta}$ in $\PGaL(3,q^2)$. This reduces the complexity of the problem significantly.

The rest of this section is devoted to the proof of Theorem \ref{BM_Onan}. We shall handle the cases of  odd and even characteristic separately. We start with some technical lemmas.
\begin{lemma}\label{BM_alpha}
Let  $\cU_{\alpha,\beta}$ be a nonclassical orhtogonal Buekenhout-Metz unital in $\PG(2,q^2)$, $q>2$.  Then $\alpha^{q+1}\ne(\lambda-\beta)^{q+1}$ and $\alpha^q y+(\lambda-\beta)y^q \neq 0$  for all $\lambda \in \F_q$ and $y \in \F_{q^2}^*$.
\end{lemma}
\begin{proof}
Assume that  $\alpha^q y+(\lambda-\beta)y^q =0$ for some $\lambda \in \F_q$ and $y\in \F_{q^2}^*$. Then raising both sides of $\alpha^q y=-(\lambda-\beta)y^q$ to the $(q+1)$-st power, we deduce that $\alpha^{q+1} y^{q+1}=(\lambda-\beta)^{q+1}y^{q+1}$, i.e., $\alpha^{q+1}-(\lambda-\beta)^{q+1} =0$.

In the case $q$ is odd, the discriminant $d=(\beta-\beta^q)^2+4\alpha^{q+1}$ of $\cU_{\alpha,\beta}$ equals
\[
d=(\beta-\beta^q)^2+4(\lambda-\beta)^{q+1}=(\beta+\beta^q-2\lambda)^2,
\]
which contradicts the fact that $d$ is a nonsquare.
	
In the case $q$ is even, write $\beta=b_0+b_1\delta$ with $b_0,\,b_1\in\F_q$, where $\delta\in\F_{q^2}^*$ with $\delta+\delta^q=1$. The quantity $v=\delta^2+\delta=\delta^{q+1}$ has absolute trace $1$ by \cite[Lemma 4.21]{Ebert2008Unitals}.   The discriminant $d=\alpha^{q+1}/(\beta+\beta^q)^2$ equals
\[
d=\frac{(\beta+\lambda)^{q+1}}{(\beta+\beta^q)^2}
=\frac{(b_0+\lambda+b_1\delta)^{q+1}}{b_1^2}=\frac{(b_0+\lambda)^2}{b_1^2}+\frac{b_0+\lambda}{b_1}+v.
\]
Taking the absolute trace on both sides, we have $\tr_{q/2}(d)=1$, which is again a contradiction. This completes the proof.
\end{proof}

For $y\in \F_{q^2}^*$, let $\ell_{y}$ be the line
\[
[y,1,0]=\{x\in\PG(2,q^2):\,x \cdot (y,1,0)=0\},
\]
where $\cdot$ is the usual Euclidean inner product. This line contains the point $P=(0,0,1)$ of $\cU_{\alpha,\beta}$. The tangent line of $\cU_{\alpha,\beta}$ at $P$ is $[0,1,0]$, so the set $B_y:=\ell_y\cap\cU_{\alpha,\beta}$ consists of $q+1$ points and is a block of the corresponding unital design. Also, the line $\ell_{\infty}=[1,0,0]$ corresponds to the block
$B_{\infty}=\{(0,r,1):r\in\F_q\}\cup\{(0,1,0)\}$.

\begin{lemma}\label{BM_line}
For each  $y \in \F_{q^2}^*$, we have
\[
B_y=\left\lbrace(x_\lambda,-x_\lambda y,1): x_\lambda =-\frac{\alpha^q y+(\lambda-\beta)y^q}{\alpha^{q+1}-(\lambda-\beta)^{q+1}},\ \lambda \in \F_q\right\rbrace \cup\{(0,0,1)\}.
\]
\end{lemma}
\begin{proof}
The point $(x,\alpha x^2+\beta x^{q+1}+r,1)$  with $x \neq 0$, $r\in\F_q$ lies in the block $B_y$ if and only if $yx+\alpha x^2+\beta x^{q+1}+r =0$. Write	$\lambda:=-rx^{-(q+1)}\in\F_q$ so that we have $y+\alpha x+\beta x^{q} =\lambda x^{q}$.
Raising  both sides to the $q$-th power, we get $y^q+\alpha^q x^q+\beta^q x =\lambda^q x$. It is now routine to deduce that $x =-\frac{\alpha^q y+(\lambda-\beta)y^q}{\alpha^{q+1}-(\lambda-\beta)^{q+1}}$.  The claim now follows.
\end{proof}

\subsection{The odd characteristic case}

Consider the following involutionary central collineation of $\textup{PGL}(3,q^2)$:
 $$\sigma: (x,y,z)\mapsto (-x,y,z).$$
It stabilizes the unital $\cU_{\alpha,\beta}$, fixes the set $B_{\infty}$ pointwise and maps $B_y$ to $B_{-y}$ for $y\in \F_{q^2}^*$ by a direct check. We now assume that the putative O'Nan configuration $\cN$ contains two lines $\ell_1$, $\ell_{-1}$ through $P=(0,0,1)$. Let $\ell'$, $\ell''=\sigma(\ell')$ be the other two lines in $\cN$, and assume that $\ell'\cap\ell''=(0,r,1)\in \cU_{\alpha,\beta}$ for some $r\in \F_q$ to be determined. The O'Nan configuration $\cN$ that we are seeking for  is fixed by the involution $\sigma$, see Figure \ref{BMO}.
\begin{figure}[htbp]
\centering
\includegraphics[width=12cm]{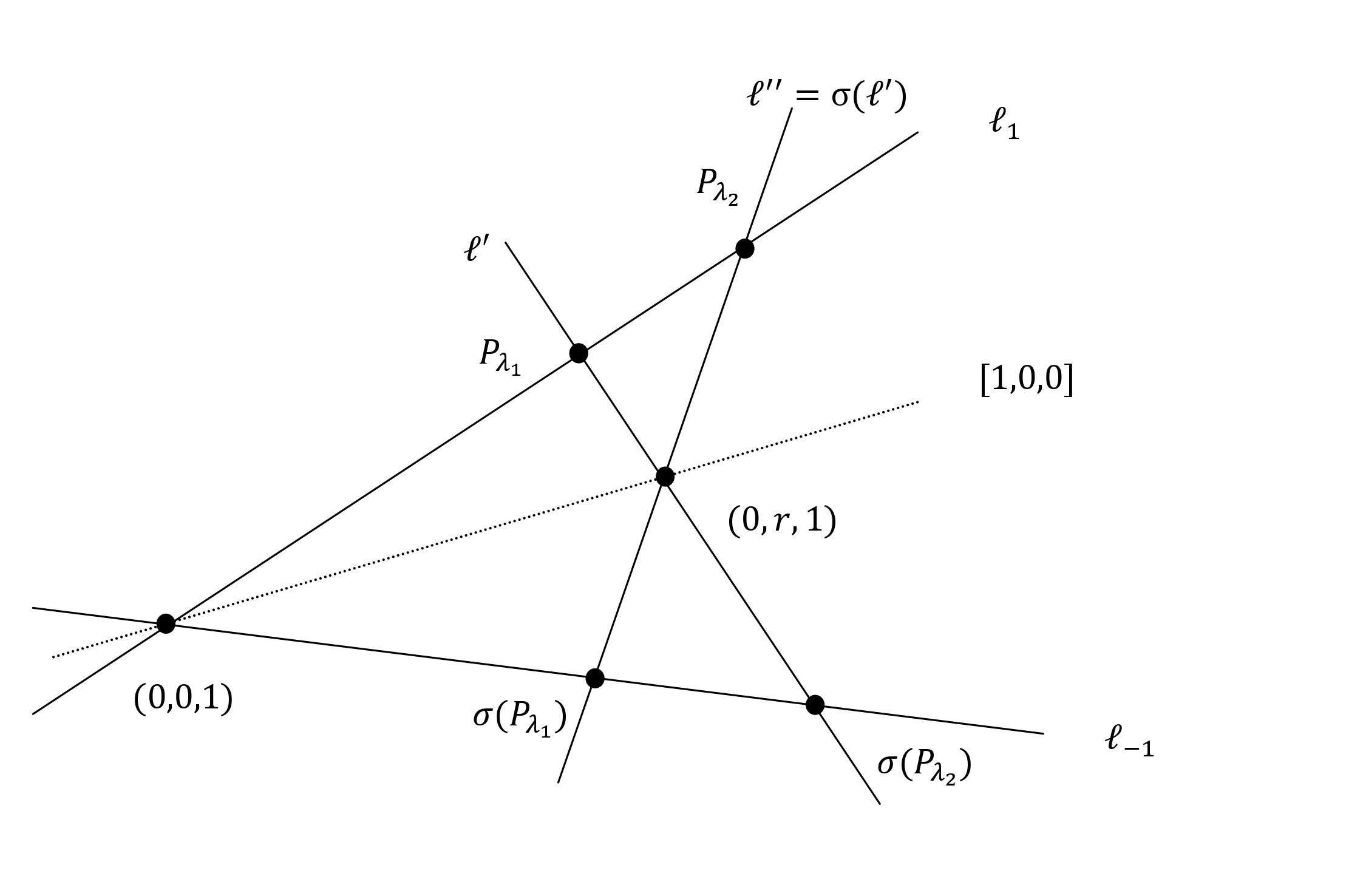}
\caption{A putative O'Nan configuration in $\cU_{\alpha,\beta}$ when $q$ is odd}\label{BMO}
\end{figure}
By Lemma \ref{BM_line}, the intersection point of $\ell_1$ and $\ell'$ is
\[
P_{\lambda_1}=(x_{\lambda_1},x_{\lambda_1},1) ,\quad x_{\lambda_1}=-\frac{\alpha^q +\lambda_1-\beta}{\alpha^{q+1}-(\lambda_1-\beta)^{q+1}}
\]
for some $\lambda_1\in\F_q$. Similarly, the intersection of $\ell_1$ with $\ell''$ is
\[
P_{\lambda_2}=(x_{\lambda_2},x_{\lambda_2},1),\quad x_{\lambda_2}=-\frac{\alpha^q +\lambda_2-\beta}{\alpha^{q+1}-(\lambda_2-\beta)^{q+1}}
\]
for some $\lambda_2 \in \F_q$ with $\lambda_1\ne\lambda_2$.  It follows that $\sigma(P_{\lambda_1})$ and $\sigma(P_{\lambda_2})$ are the intersection points of $\ell''$ and $\ell'$ with $\ell_{-1}=\sigma(\ell_1)$ respectively. Therefore,
\[
 \ell'=\langle (x_{\lambda_1},x_{\lambda_1},1) ,(-x_{\lambda_2},x_{\lambda_2},1)\rangle,\quad \ell''=\langle (x_{\lambda_2},x_{\lambda_2},1) ,(-x_{\lambda_1},x_{\lambda_1},1)\rangle.
\]
We directly compute that $\ell'\cap \ell''=(0,r,1)$ with
\[
 r=\frac{2x_{\lambda_1}x_{\lambda_2}}{x_{\lambda_1}+x_{\lambda_2}}.
\]
The condition $r\in\F_q$, i.e., $r=r^q$, amounts to $x_{\lambda_1}^{q+1}\ (x_{\lambda_2}- x_{\lambda_2}^q)+(x_{\lambda_1}-x_{\lambda_1}^q)\ x_{\lambda_2}^{q+1} =0 $.  Plugging in the expressions for $x_{\lambda_1},\,x_{\lambda_2}$, we deduce that
\begin{equation}\label{ON_odd}
 \left(h(\lambda_1)g(\lambda_2)+h(\lambda_2)g(\lambda_1)\right) (\alpha^q-\alpha+\beta-\beta^q) =0,
\end{equation}
with $g(X)= -X^2+\alpha^{q+1}-\beta^{q+1}$  and  $h(X)= (X+\alpha^q-\beta)^{q+1}.$

Since $d=(\beta-\beta^q)^2+4\alpha^{q+1}$ is a nonsquare in $\F_q$, we have $\alpha^q-\alpha+\beta-\beta^q \neq 0$: otherwise, $d=(\alpha+\alpha^q)^2$ would be a square.
We thus can get rid of the second factor on the left hand side in Eqn. \eqref{ON_odd}. Observe that $g(x) \neq 0$ and $h(x)\neq 0$ for any $x\in\F_q$ by Lemma \ref{BM_alpha} with $y=1$. Now we define
$$\kappa(x):=\frac{h(x)}{g(x)}, \quad x \in\F_q.$$
Observe that $\kappa(x)$ is nonzero in $\F_q$ for each $x\in\F_q$. Eqn. \eqref{ON_odd} now reduces to
\begin{equation}\label{eqn_l1l2}
\kappa(\lambda_1)+\kappa(\lambda_2)=0.
\end{equation}
Observe that if Eqn. \eqref{eqn_l1l2} holds, then necessarily $\lambda_1\ne\lambda_2$: otherwise, we have $\kappa(\lambda_1)=\kappa(\lambda_2)=0$, which is impossible.

To summarize, we have established the following result.
\begin{lemma}
Suppose that $\alpha\ne 0$ and $q$ is ood.
If there are elements $\lambda_1,\,\lambda_2$ in $\F_q$ such that $\kappa(\lambda_1)+\kappa(\lambda_2)=0$, then there is an O'Nan configuration $\cN$ in the unital $\cU_{\alpha,\beta}$.
\end{lemma}

By Lemma \eqref{BM_pj}, the unitals $\cU_{\alpha,\beta}$ and $\cU_{\alpha,\beta+u}$ are projectively equivalent for any $u\in\F_q$. Therefore, by replacing $\beta$ with $\beta+u$ for some $u\in\F_q$ if necessary, we assume that
\begin{equation}\label{eqn_assumptr}
\tr_{q^2/q}(\alpha-\beta) \ne 0.
\end{equation}
Set $K:=\{\kappa(x):\,x\in\F_q\}$. Take $k \in K$,  and consider the  equation $ h(X)-kg(X)=0$, i.e.,
\begin{equation} \label{ON_odd1}
(1+k)X^2-\big(\tr_{q^2/q}(\alpha-\beta)\big) X+(\alpha^q-\beta)^{q+1}-k(\alpha^{q+1}-\beta^{q+1})=0.
 \end{equation}
This is a polynomial of degree at most $2$, and has degree at least $1$ by the assumption \eqref{eqn_assumptr}. Since $k$ is in $K$, Eqn. \eqref{ON_odd1} has $1$ or $2$ solutions in $\F_q$. Therefore, we have $|K|\geq \lceil\frac{q}{2}\rceil=\frac{q+1}{2}$. It follows that
\[
|K\cap -K|\ge|K|+|-K|-q\ge1,
\]
where $-K=\{-k:\,k\in K\}$. Therefore, there is at least one pair $(\lambda_1,\,\lambda_2)\in\F_q^2$ such that $\kappa(\lambda_1)+\kappa(\lambda_2)=0$.

To summarize, we have now established the existence of O'Nan configurations of the form in Figure \ref{BMO} in the nonclassical unital $\cU_{\alpha,\beta}$'s in the case $q$ is odd.

\subsection{The even characteristic case}
We now consider the case $q$ is even, $q>2$. Take $\delta\in\F_{q^2}^*$ with $\delta+\delta^q=1$, and set $v=\delta^{q+1}\in\F_q$ which has absolute trace $1$ by \cite[Lemma 4.21]{Ebert2008Unitals}. Observe that there is some freedom in the choice of $\delta$ which we will explore below. The unital $\cU_{\alpha,\beta}$ is nonclassical implies that $\alpha\ne 0$ and $\beta\not\in\F_q$. It is routine to see that there are elements $s\in\F_{q^2}^*$, $f\in\F_q^*$ and $u\in\F_q$ such that $\alpha'=\alpha s^2 f \in\F_q^*$ and $\beta'=(\beta s^{q+1}+u)f= \delta$.
 By Lemma \ref{BM_pj}, we assume without loss of generality that the parameters of the unital $\cU_{\alpha,\beta}$ in consideration satisfies that $\alpha=a\in \F_q^*, \,\beta =\delta$.
\begin{lemma}
There exists $\eta\in\F_{q^2}$ such that $\eta+\eta^q=1$,  $\tr_{q/2}\left(\frac{a^2}{\eta+\eta^2}\right)=1$ for a fixed $a\in\F_q^*$.
\end{lemma}
\begin{proof}
The set $\{x+x^2:\,x\in\F_{q^2},\,x+x^q=1\}$ has size $q/2$ and does not contain $0$. Meanwhile, $\tr_{q/2}(a^2X)=0$ has $q/2-1$ nonzero solutions, and the claim follows by comparing the two sizes.
\end{proof}
Take $\delta'\in\F_{q^2}$ such that  $\delta'+\delta'^q=1$,  $\tr_{q/2}\left(\frac{a^2}{\delta'+\delta'^2}\right)=1$. Then $\delta'$ and $\delta$ differ by an element of $\F_q$, and the unital $\cU_{a,\delta}$ is isomorphic to $\cU_{a,\delta'}$ by Lemma \ref{BM_pj}. Therefore, we only need to consider the unitals $\cU_{\alpha,\beta}$ whose parameters satisfy that
\begin{equation}\label{ass_BMe}
\alpha=a\in\F_q^*,\quad\beta=\delta
\end{equation}
such that $\tr_{q/2}\left(\frac{a^2}{v}\right)=1$, where $v=\delta^{q+1}=\delta^2+\delta$.

Take the following involution in $\PGaL(3,q^2)$:
\begin{equation}\label{EBM_inv}
  \phi:\,(x,y,1)\mapsto (x^q,y^q,1).
 \end{equation}
It stabilizes the unital $\cU_{\alpha,\beta}$, fixes the block $B_{\infty}$ pointwise and maps $B_{y}$ to $B_{y^q}$ for $y\in \F_{q^2}^*$.
Suppose that the putative O'Nan configuration $\cN$ contains two lines $\ell_{\delta},\,\ell_{\delta^{q}}$ through $P=(0,0,1)$, and assume that the other two lines $\ell',\,\ell''$  in $\cN$ satisfy that $\ell'=\phi(\ell')$, $\ell''=\phi(\ell'')$ and $\ell'\cap\ell''=(0,r,1)\in\cU_{\alpha,\beta}$ for some $r\in \F_q\setminus\{0\}$. Please notice the distinction with the odd characteristic case.

Let  $P_{\lambda}=(x_{\lambda},\delta x_{\lambda},1)$ be a point in the block $B_{\delta}$ for  some $\lambda \in\F_q$, where
\[
x_{\lambda}=\displaystyle\frac{a\delta+(\lambda+\delta)\delta^q}{a^2+(\lambda+\delta)^{q+1}}
=\frac{\lambda+v+(a+\lambda)\delta}{a^2+\lambda^2+\lambda+v}.
\]
The point $\phi(P_{\lambda})$ lies in $B_{\delta^q}=\phi(B_\delta)$. We deduce that the intersection point of $\ell_{\infty}$ and $P_{\lambda}\phi(P_{\lambda})$ is
\[
\left(0,\frac{x_{\lambda}^{q+1}}{x_{\lambda}+x_{\lambda}^q},1\right)
=\left(0,\frac{\lambda^2v+\lambda(a+v)+av+v^2+a^2v}
{(a^2+\lambda^2+\lambda +v)(\lambda +a )},1 \right).
\]
If we can find two distinct solutions $\lambda_1,\,\lambda_2\in\F_q$ for the equation
\begin{equation}\label{EBM_eq}
r^{-1}=G(X):=\frac{(a^2+X^2+X +v)(X +a )}{X^2v+X(a+v)+av+v^2+a^2v}
 \end{equation}
for some $r\in\F_q^*$, then the lines $B_{\delta}$, $B_{\delta^q}$, $\ell'=P_{\lambda_1}\phi(P_{\lambda_1})$ and $\ell''=P_{\lambda_2}\phi(P_{\lambda_2})$ form an O'Nan configuration of the prescribed form in the unital $\cU_{a,\beta}$.

We note that $G(x)$ lies in $\F_q$ for all $x\in\F_q$, since all its coefficients are in $\F_q$. We compute that $G(v)=\frac{a+v}{v}$, and deduce from $G(Z+a)=G(v)$ that
\[
(Z+a+v)\left(Z^2+Z+\frac{a^2}{v}+v\right)=0.
\]
Since $\tr_{q/2}\left(\frac{a^2}{v}+v\right)=1+1=0$, the second factor has two distinct solutions. Therefore, if $a+v\ne 0$, then take $r=\frac{v}{v+a}$ and $r^{-1}=G(X)$ has at least two distinct solutions in $\F_q$. Finally, we observe that $a=v$ can not occur: the discriminant $d=a$ has absolute trace $0$ while $v$ has absolute trace $1$.

To conclude, we have established the existence of O'Nan configurations of the prescribed form in the nonclassical unital $\cU_{\alpha,\beta}$'s in the even characteristic case.

\section{O'Nan configurations in Buekenhout-Tits unitals}
In this section, we establish the following result.
\begin{thm}\label{BT}
Each Buekenhout-Tits unital in $\PG(2,q^2)$ contains an O'Nan configuration.
\end{thm}
Let $\cU_{T}$ be the Buekenhout-Tits unital in $\PG(2,q^2)$ as defined in Eqn. \eqref{def_BT}, where $q=2^m$, $m$ odd and $m>1$. Take $\delta\in\F_{4}$ such that $\delta^2+\delta+1=0$. Then $1,\delta$ from a basis of $\F_{q^2}$ over $\F_q$ since $m$ is odd. We use this $\delta$ in the definition of $\cU_T$ as in Eqn. \eqref{def_BT}. 
Recall that 
\[
f(x,y)=x^{\tau+2}+xy+y^{\tau},\quad\tau=2^{(m+1)/2}.
\]
For $y\in\F_{q^2}^*$ such that the line $\ell_{y}:=[y,1,0]$ is a secant line to $\cU_T$, we define the block $B_y $ to be the intersection of $\cU_{T}$ with the line $\ell_{y}:=[y,1,0]$. Also, it is easy to see that  the line $\ell_{\infty}=[1,0,0]$   corresponds to the block $B_{\infty}=\{(0,r,1): r\in \F_q\} \cup\{(0,1,0)\}$.

It is tempting to follow the same ideas in Section 3 and look for O'Nan configurations that are stabilized by an involution in the stabilizer of $\cU_T$ in $\PGaL(3,q^2)$. However, this is impossible after some short analysis. Therefore, we have to change our strategy. Recall that there is no O'Nan configuration that contains the special point $P_{\infty}=(0,0,1)$ by Lemma \ref{Spoint}.
\begin{lemma}\label{BT_c}
Suppose that $q=2^m$ with $m$ odd and $m>1$. Then there exists $c\in \F_q^*$ such that $\tr_{q/2}(c^{\tau+2}+c+1)=0$.
\end{lemma}
\begin{proof}
The quadratic form $Q(x)=\tr_{q/2}(x^{\tau+2})$ is nondegenerate on the $\F_2$-linear vector space ($\F_q,\,+$) by direct check.  The quadric $Q(x)=0$ is thus not contained in  the hyperplane $\tr_{q/2}(x)=0$. Therefore, there exists $c\in\F_q$ such that $\tr_{q/2}(c^{\tau+2})=0$ and $\tr_{q/2}(c)=1$, and it has the desired property.
\end{proof}

Let $\cN$ be  a putative O'Nan configuration in $\cU_T$ that contains $(0,0,1)$, $\ell_{\infty}$ and a line $\ell_c$ through $(0,0,1)$, where $c\in\F_q^*$ is chosen such that
\[
\tr_{q/2}(c^{\tau+2}+c+1)=0.
\]
Such a $c$ exists by the preceding lemma. Further assume that the other two lines of $\cN$, $\ell,\, \ell'$, intersect in a point  $P=(1,v+\delta,1)$ with $v=\frac{1}{c^{\tau}+1}(c^2+c^{1-\tau}+c^{\tau+1})$. Write
\[
\ell\cap\ell_\infty=(0,r_1,1),\quad \ell'\cap\ell_\infty=(0,r_2,1), 
\]
where $r_1,\,r_2$ are two distinct elements of $\F_q^*$. Both points lie in the block $B_\infty$. The putative O'Nan configuration $\cN$ is depicted in Figure \ref{BTO}.
\begin{figure}[htbp]
\centering
\includegraphics[width=12cm]{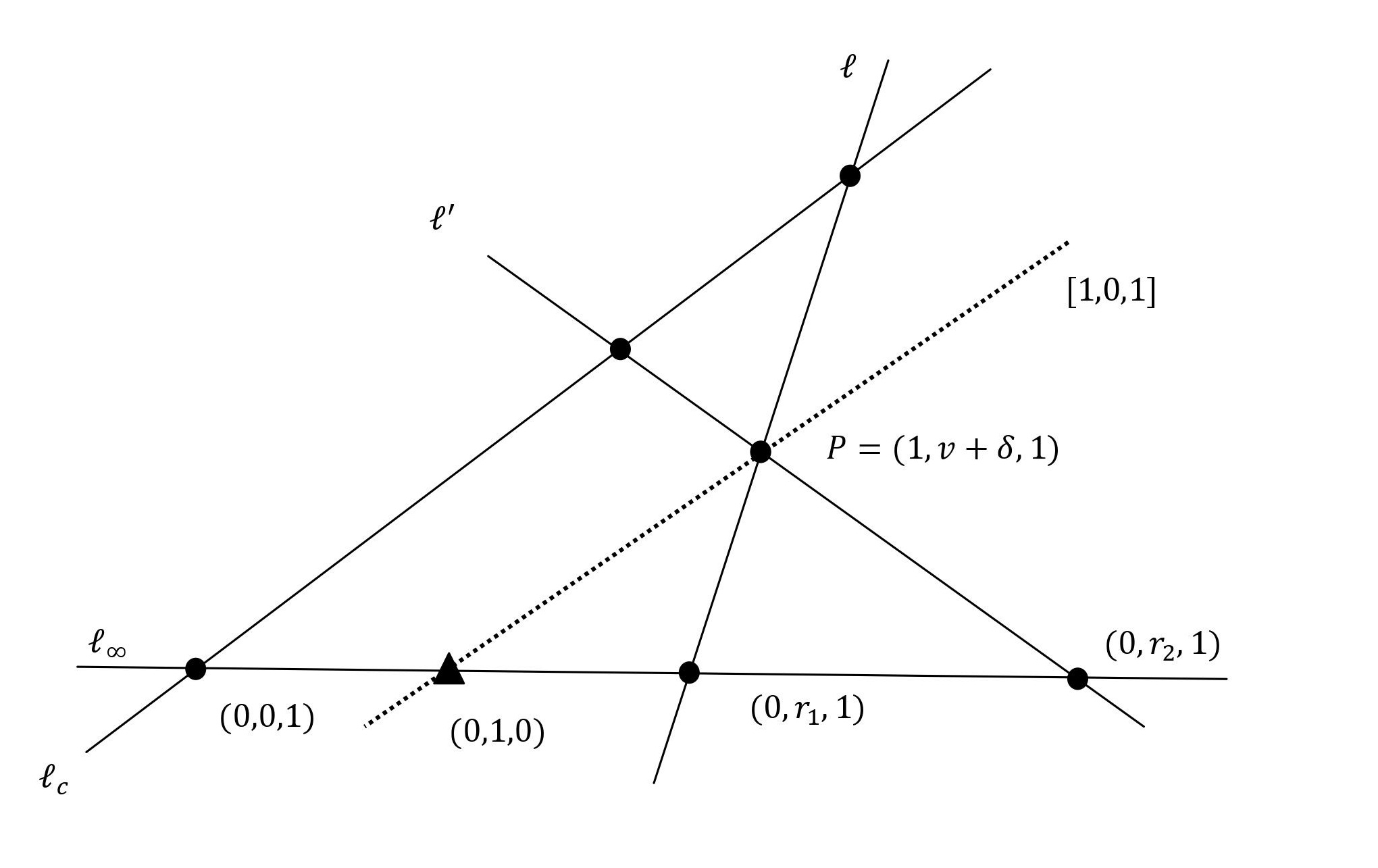}
\caption{The putative O'Nan configuration in $\cU_T$}\label{BTO}
\end{figure}
To construct an O'Nan configration $\cN$ of the prescribed form, we need to find two distinct elements $r_1,\,r_2$ in $\F_q^*$ such that the corresponding lines $\ell$, $\ell'$ both intersect $\ell_c$ in a point on $\cU_T$.

Take an element $r\in\F_q^*$, and consider the line $\ell''=[r+v+\delta,1,r]$ which contains both $(0,r,1)$ and $P=(1,v+\delta,1)$. We deduce that the intersection point $P'$ of $\ell''$ and $\ell_c$ is
$$P'=(r,rc,r+c+v+\delta).$$
For $x,\,y\in\F_q$, not both zero, we have $(x+y\delta)^{-1}=\frac{(x+y)+y\delta}{H(x,y)}$ with $H(x,y)=x^2+xy+y^2$. We thus rewrite $P'$ as
\[
P'=\left(\frac{r(r+c+v+1+\delta)}{H(r+c+v,1)},\, \frac{rc(r+c+v+1+\delta)}{H(r+c+v,1)},\,1\right).
\]
It lies on the unital $\cU_T$ if and only if
\begin{equation}\label{BT_cond}
  f\left(\frac{r(r+c+v+1)}{H(r+c+v,1)},\frac{r}{H(r+c+v,1)}\right)=\frac{rc}{H(r+c+v,1)}.
\end{equation}
Applying Lemma \ref{fxy} with $x=\frac{r(r+c+v+1)}{H(r+c+v,1)}$ and $y=\frac{r}{H(r+c+v,1)}$, we have that the left hand side of Eqn. \eqref{BT_cond} equals
$f\left(c^{-1}, c^{-1}(r+c+v+1)\right)^{-1}$. Therefore, the condition $\eqref{BT_cond}$ is equivalent to
\begin{equation*}\label{BT_cond1}
  f\left(\frac{1}{c},\frac{r+c+v+1}{c}\right)=\frac{H(r+c+v,1)}{rc}.
\end{equation*}
Expanding and simplifying it, we have 
\begin{equation}\label{BT_eq}
    \frac{r^{\tau}}{c^{\tau}}+\frac{c+1}{c^2}r+
    \left(\frac{1}{c^{\tau+2}}+\frac{v+1}{c^2}+1+\frac{v^{\tau}+1}{c^{\tau}}\right)+\frac{H(c+v,1)}{rc}=0.
\end{equation}
To summarize, we have proved the following result.
\begin{lemma}
If Eqn. \eqref{BT_eq} have two distinct nonzero solutions in $\F_q$, then we have a O'Nan configuration in $\cU_T$ as depicted in Figure \ref{BTO}. 
\end{lemma}
\begin{proof}
Take $r_1,\,r_2$ to be two distinct nonzero solutions of Eqn. \eqref{BT_eq}.
\end{proof}
Recall that $v=\frac{1}{c^{\tau}+1}(c^2+c^{1-\tau}+c^{\tau+1})$. Now define
\begin{equation*}\label{f1}
f_1(X):=\frac{X^{\tau}}{c^{\tau}}+\frac{X}{c^2}+\left(\frac{1}{c}+1\right)^{\tau+2}+\frac{v}{c^2}+\frac{v^{\tau}}{c^{\tau}}+ \frac{1}{c^{\tau+1}}.
\end{equation*}
It is tedious but routine to check that  Eqn. \eqref{BT_eq} can be rewritten in the form
\begin{equation*}\label{BT_eq1}
f_1(r)+\frac{c}{r}\left(\big(f_1(r)\big)^{\tau}+\frac{1}{c^\tau}f_1(r)\right)=0.
\end{equation*}
\begin{lemma}
The polynomial $f_1(X)=0$ has two nonzero solutions in $\F_q$.
\end{lemma}
\begin{proof}
Set $Y=c^\tau X$, and we rewrite $c^{\tau+2}f_1(X)=0$ as follows
\[
Y^\tau+Y+A=0,\quad A=(1+c)^{\tau+2}+c+c^\tau v+c^2v^\tau.
\] 
The kernel of the linear map $y\mapsto y^\tau+y$ is $\F_2$, so the above equation has either $0$ or $2$ solutions. Notice that the equation $Z^2+Z+A=0$ has a solution if and only if $\tr_{q/2}(A)=0$. Then $Y=z^\tau+z$ (or $z^\tau+z+1$) is one of the solutions of $Y^\tau+Y+A=0$ if $z^2+z+A=0$ for some $z\in\F_q$ . Hence we compute that
\begin{align*}
\tr_{q/2}(A)&=\tr_{q/2}\left((1+c)^{\tau+2}+c+c^\tau v+c^2v^\tau\right)\\
  &=\tr_{q/2}(1+c^{\tau+2}+c)=\tr_{q/2}(c^{\tau+1}+c+1)=0.
\end{align*}
Here we have used the fact that $\tau^2\equiv 2\pmod{q-1}$.  It remains to show that $A\ne 0$. And we prove it by contradiction. Substituting the value of $v$ into $A=0$, we have
\begin{equation}\label{Aeq}
(1+c)^{\tau+2}+\frac{c^{\tau+2}}{c^{\tau}+1}+c^{\tau+1}+c+\frac{c^{2\tau+2}}{c^{2}+1}+c^{\tau+2}+c^{\tau}=0
\end{equation}
Multiplying with $\frac{(1+c)^{\tau+2}}{c^{\tau+2}}$ and simplifying it, we have 
$$c^{-\tau-2}+c^{2-\tau}+c^{-2}+c^{-1}+1+c+c^{\tau-1}+c^{\tau}+c^{\tau+1}+c^{2\tau}=0.$$
Then taking the absolute trace function with the condition $\tr_{q/2}(c^{\tau+1}+c+1)=0$,  we have $\tr_{q/2}(c^{-\tau-2})=0$. And for the case Multiplying with $\frac{(1+c)^{\tau+2}}{c^{2\tau+2}}$ on the both sides of Eqn. \eqref{Aeq}, we have $\tr_{q/2}({c^{-\tau-2}}+1)=0$, which is a contradiction. This completes the proof. 
\end{proof}

To summarise, we have now completed the proof of Theorem \ref{BT}.

\section{Conclusion}

Recently, Korchm\'{a}ros, Sicilliano and Sz\"{o}nyi \cite{embed2018} proved that there is a unique embedding of the classical unital in $\PG(2,q^2)$. The basic idea is that the existence of certain configurations consisting of points and blocks implies the existence of isomorphisms between blocks. In this paper, we are inspired by their ideas, and search for O'Nan configurations of a particular form in Buekenhout unitals in $\PG(2,q^2)$.  In the case of an orthogonal Buekenhout-Metz unital, the blocks (i.e the intersection of the unital with secant lines) form  Baer sublines, and we have established the existence of O'Nan configurations which are fixed by an involution in the stabilizers of the unital in $\PGaL(3,q^2)$.   In the case of a Buekenhout-Tits unital, we establish the existence of O'Nan configurations that contain a fixed Baer subline. Our results provide  evidence to the truth of Piper's conjecture that the absence of O'Nan configurations characterizes the classical unitals.

\section*{Acknowledgement}\
This work was supported by National Natural Science Foundation of China under
   Grant No. 11771392.

\section*{Reference}
\scriptsize
\setlength{\bibsep}{0.5ex}  

%

\end{document}